\documentclass{amsart}

\usepackage{amsmath}
\usepackage{amsthm}
\usepackage{amssymb}
\usepackage[all]{xy}

\newtheorem{theorem}{Theorem}[section]

\newtheorem{proposition}[theorem]{Proposition}

\theoremstyle{definition}
\newtheorem{definition}[theorem]{Definition}

\newcommand{\comp}{\circ}

\newcommand{\To}{\longrightarrow}
\renewcommand{\to}{\longrightarrow}

\renewcommand{\ker}{\ensuremath{\mathrm{ker}}}

\newcommand{\A}{\ensuremath{\mathbf{A}}}
\newcommand{\X}{\ensuremath{\mathbf{X}}}
\newcommand{\B}{\ensuremath{\mathbf{B}}}

\DeclareMathOperator{\Aut}{Aut}

\newcommand{\Pt}{\ensuremath{\mathsf{Pt}}}

\newcommand{\W}{\ensuremath{\mathsf{W}}}

\newcommand{\Set}{\ensuremath{\mathsf{Set}}}
\newcommand{\Grp}{\ensuremath{\mathsf{Grp}}}
\newcommand{\Gpd}{\ensuremath{\mathsf{Gpd}}}

\newcommand{\Simp}{\ensuremath{\mathsf{Simp}}}
\newcommand{\Cat}{\ensuremath{\mathsf{Cat}}}

\renewcommand{\a}{\alpha}
\renewcommand{\b}{\beta}

\begin{document}
\title{Crossed-modules and Whitehead sequences}
\author[N. Martins-Ferreira]{Nelson Martins-Ferreira}
\address{School of Technology and Management-ESTG\\
Centre for Rapid and Sustainable Product Development-CDRSP\\
Polytechnic Institute of Leiria\\
P-2411-901, Leiria, Portugal
}
\email{martins.ferreira@ipleiria.pt}

\keywords{Whitehead sequence, action-system, organic morphism, exact-right-patch, protomodular category, crossed-module, simplicial object, internal groupoid}

\subjclass[2010]{Primary 18A05; Secondary 20J15, 20L05}

\thanks{Thanks are due to M. Sobral and T. Van der Linden for helpful comments on the text. Research supported by IPleiria/CDRSP, FCT Post-doc grant SFRH/BPD/43216/2008 at CMUC and FCT project PTDC/EME-CRO/120585/2010}

\begin{abstract}
We introduce the notion of Whitehead sequence which is defined for a base category together with a system of abstract
actions over it. In the classical case of groups and group actions the  Whitehead sequences are precisely
the crossed-modules of groups. For a general setting we give sufficient conditions for the existence of a
categorical equivalence between internal groupoids and Whitehead sequences.
\end{abstract}

\maketitle

This work is dedicated to  Ronnie Brown on the occasion of his 80th anniversary.

\section{Introduction}

This work may be seen as the continuation of the project initiated in 1982 by T. Porter \cite{Porter82} in order
to generalize the so-called Brown-Spencer result \cite{Brown-Spencer} from groups to other structures. The
Brown-Spencer result establishes a categorical equivalence between crossed modules of groups and internal
groupoids in the category of groups, an important result connecting two types of objects, apparently with a
very different nature. As a consequence, this result has significant applications in homotopy theory, homology,
cohomology, K-theory and higher dimensional categorical algebra, among others.

Over the last three decades many authors studied this specific problem.
The original result, although already known, was first published in 1976 \cite{Brown-Spencer}. In 1982
J.-L. Loday generalizes it to  higher-dimensions \cite{Loday} by introducing the notion of Cat-n-group.
During the 80's much work was done, either as applications of the original result or as generalizations of it,
especially in categories of groups with operations, as it can be seen for example in \cite{Porter87}.
In the 90's R. Brown and his School were still active in this area as one can see in \cite{Brown:Groupoids}
and its references, as well as  several other authors. For example internal categories and internal groupoids
started to be exhaustively studied, first in the context of Mal'tsev categories and later in the context of
semiabelian categories. This work culminated with the notion of internal crossed module by G. Janelidze,
see \cite{Janelidze} which also contains some historical notes.

The main motivation for the present work was the possibility of moving from the category of internal actions,
defined in the context of a semiabelian category, to a more general context of categories and functors such as
the one we introduce in Section \ref{sec:defs}, in which $\B$ is any pointed category while $\A$ can be interpreted
as a category of abstract actions on $\B$.

After a close analysis of some of the results obtained during the last three or four decades we concluded that
many of the generalizations of the notion of crossed module were obtained by calculating simpler descriptions of
internal groupoids. The perspective that we have adopted in this work is somehow different. We consider a general
system   in which a certain sequence of two morphisms without any further assumptions is considered. We call it a
Whitehead sequence. Accordingly, we define a crossed module as a Whitehead sequence to which an internal groupoid
structure can be associated in a  canonical manner, an idea that we will make precise later.

Consider a system of functors and categories displayed as
\begin{equation}\label{system}
\vcenter{\xymatrix@=3em{\A \ar@<1.5ex>[r]^-{I} \ar@<-1.5ex>[r]_-{J} & \B \ar[l]|-{G} }}
\end{equation}
and such that $IG=1_{\B}=JG$. A sequence of morphisms in $\A$ of the form
\[
\vcenter{\xymatrix{GJA \ar[r]^-{v} & A \ar[r]^-{u} & GIA }}
\]
is called a \emph{Whitehead sequence} whenever
\[
I(u)=1_{IA},\ J(v)=1_{JA}, \  I(v)=J(u).
\]

Our main goal here is to find reasonable conditions under which we have an equivalence of categories
\begin{equation}
\W(\A) \sim \Gpd(\B)
\end{equation}
between the category of Whitehead sequences in $\A$ and the category of internal groupoids in $\B$,
the guiding example being the case where $\B$ is the category of groups and $\A$ is the category of group
actions on groups. The functors $I$ and $J$ are the obvious projections (see Section \ref{sec:examples})
while $G$ gives the action by conjugation \cite{Porter82}.
The functor $G$ has a left adjoint, $F$, which corresponds to the well known construction of the semidirect product
in groups (see also \cite{Bourn-Janelidze:Semidirect} and \cite{Janelidze}). In this case a Whitehead sequence is
precisely a crossed module.

A crossed module, as introduced by J.H.C. Whitehead \cite{Whitehead-CHII},
in the category of groups
consists of a pair $(A,h)$ in which $A \in \A$ is an action (the group $IA$ acts on the group $JA$) and $h:JA \to IA$ is a group homomorphism such that there is a Whitehead sequence
\[
\vcenter{\xymatrix{GJA \ar[r]^-{v} & A \ar[r]^-{u} & GIA }}
\]
with $J(u)=h=I(v)$.

This notion of crossed module was already presented in \cite{Porter82}. Here we illustrate the general system of
categories and functors (\ref{system}) and  motivate the definition of Whitehead sequence which,  in the particular
case where $\B$ is the category of groups and $\A$  the category of group actions,  gives the classical notion of a
crossed module.

First we give additional conditions on the general system of categories and functors  (\ref{system})
in which  $\B$ is a pointed category while $\A$ (under some reasonable conditions) is to be understood as
a category of actions. More specifically, an object $A$ in $\A$ is considered as an action of the object
$IA$ on the object $JA$ in $\B$ and if $B$ is an object in $\B$ then $G(B)$ is considered as an action
(by conjugation) of $B$ onto itself.

One of the important aspects of this construction is that  we can always define the notion of Whitehead sequence
as a triple $(A,u,v)$, in $\A$, of the form
\[
\vcenter{\xymatrix{GJA \ar[r]^-{v} & A \ar[r]^-{u} & GIA }}
\]
such that
\[
I(u)=1_{IA},\ J(v)=1_{JA}, \  I(v)=J(u),
\]
and the question is: when does it make sense to call such a sequence a crossed module? One possible answer is:
whenever it has an associated groupoid structure.

Next we describe the main ideas that lead us to the notion of category of actions we introduce here.

Concerning the one dimensional case, we assume that $I$ and $J$ are jointly faithful. This restriction means that
an action, in general, can be understood as a triple $A=(X,\xi,B)$ where $\xi$ is some kind of structure defined
on $B$ and $X$, while a morphism $f:A \to A'$ is always a pair of morphisms $f_1:JA \to JA'$ and $f_2:IA \to IA'$
in $\B$ satisfying some compatibility conditions with respect to the structures involved.
With this restriction we have that a Whitehead sequence is determined by a pair $(A,h)$ where $A$ is an action,
i.e. an object in $\A$, $h:JA \to IA$ is a morphism in $\B$, and the existence of $u$ and $v$ as in the definition
of Whitehead sequence becomes a property of $A$ and $h$, giving, in the case of groups \cite{Porter82}, the
celebrated conditions for a crossed module (equations  (\ref{xmod-1}) and (\ref{xmod-2})
of Section \ref{sec:examples}). Note that equation (\ref{xmod-1}) is equivalent to the existence of $u$, while the
equation (\ref{xmod-2}) is equivalent to the existence of $v$.

In  higher dimensions, to assume the above restriction is too much. We will often be interested in considering that
the 2-cells are also involved and in that case a morphism between actions can be a triple $f=(f_1,f_2,f_3)$ where
$f_1$ and $f_2$ are still morphisms as above but $f_3$ may be a 2-cell linking the two structures. This is what
happens in the case of categorical groups \cite{Categorical-XMod}. However, also in this case, the 2-cells involved
are determined up to equivalence. In the following we are going to consider only the one-dimensional level.
Nevertheless, the theory of action-systems presented here is delineated having in mind its application in a
two-dimensional setting.
%
%

This work is organized as follows. In Section \ref{sec:defs} we introduce the setting and give the basic definitions.
 A (right) patch is a jointly epimorphic cospan with the property that there exists a retraction of the right inclusion.
 If this retraction is the cokernel of the left inclusion then we speak of an exact (right) patch
 (Definitions \ref{def:patch} and \ref{def:exactpatch}). A patch is stable if the pullback of its retraction along
 any morphism exists and is a (right) patch (Definition \ref{def:stablepatch}). We briefly recall the well-known
 concepts of cartesian morphism and fibration. With respect to an ordered pair, $(I,J)$, of functors we define the
 notion of organic morphism (Definition \ref{def:organic}): a morphism $f\colon{A\to E}$ is organic (or $(I,J)$-organic)
  if $IE\cong JE$ and the two components $I(f)$ and $J(f)$ give rise to an exact patch.

The notion of a system such as (\ref{system}) that models a system of actions over the base category $\B$ is given
in  Definition \ref{def:actionsystem} and it is called an action-system of $\A$ over $\B$. One of the key ingredients
of the definition is what we call the L-condition (in honour of Jean Louis Loday \cite{Loday}). We point out that this
condition (see Definition \ref{def:Loday}) in the context of a semiabelian category is precisely the so-called
\emph{Smith is Huq} condition \cite{MF+VdL:12}.

In Section \ref{sec:examples} we present the main examples that have been the guiding lines for this work.
If $\B$ is a pointed category with pullbacks along split epimorphisms and binary coproducts then we can always
consider the two extreme cases. The first case is displayed as
\begin{equation}
\vcenter{\xymatrix@=3em{\A=\B \times \B \ar@<1.5ex>[r]^-{\pi_2} \ar@<-1.5ex>[r]_-{\pi_1} & \B \ar[l]|-{\Delta}}},
\end{equation}
while the second one is obtained by taking  a system such as the one displayed in (\ref{system}) with
$IG=1_{\B}=JG$, in which $\A$  is a subcategory of $\Pt(\B)$ consisting of those split epimorphisms in $\B$
\[
\xymatrix{X  \ar[r]^{k} &  Y \ar@<0.5ex>[r]^{p} & B \ar@<.5ex>[l]^{s}}
\]
such that the kernel $k$ and the section $s$ are jointly epimorphic. In order to have
the functor $G$ well defined with $G(B)$ the canonical split epimorphism
\[
\xymatrix{B  \ar[r]^-{\left\langle 0, 1\right\rangle} &  B\times B \ar@<0.5ex>[r]^-{\pi_2} & B \ar@<.5ex>[l]^-{\left\langle 1, 1\right\rangle},}
\]
for every object $B$ in $\B$, the pair of morphisms
$(\left\langle 0, 1\right\rangle, \left\langle 1, 1\right\rangle)$ must be jointly epimorphic. Concrete examples can be constructed by taking $\B$ equipped with a  forgetful (faithful and preserving binary products)  functor into the category of algebras with one constant and one binary operation, say $(X,0,/)$, satisfying the
 conditions
\begin{align}
    x/y=x'/y \implies x=x' \\
    x/x=y/y
\end{align}
where the homomorphisms are the mappings $f:X \to X'$ such that
\begin{align}
    f(x/y)=f(x)/f(y) \\
    f(0)=0.
\end{align}

In this case, the left adjoint to $G$, which is comparable to the semidirect product construction in the monadic
approach of internal actions, is simply the projection of the middle object of a split extension. Some attempts
were done in order to find a categorical notion of semidirect product (see for example \cite{Berndt}). We believe
that, in the setting of an action-system (\ref{system}) as we proposed in this paper, the notion of semidirect
product for an object $A$ in $\A$ is the object $F(A)$ in $\B$, with $F$ the left adjoint of the functor $G$,
when it exists.

 Our main result is presented in Section \ref{sec:main result}. It gives sufficient conditions  to have the
 desired categorical equivalence between Whitehead sequences and internal groupoids. This result relies on
 several other more technical results, such as a simplicial construction
 (Proposition \ref{prop:simplicialconstruction}), or an induced functor from certain kind of Whitehead sequences
 into the category of internal categories (Theorem \ref{thm:1}) which are developed on Sections \ref{sec:simp}
 and \ref{sec:cat}.

Finally, in Section \ref{sec:conclusion}, we present the case when the category $\B$ is pointed and protomodular.

\section{Basic definitions and properties}\label{sec:defs}

Let $\B$ be a pointed category.

\begin{definition}\label{def:patch}
A (right) \textbf{patch} in $\B$ is a cospan \[\xymatrix{X\ar[r]^k&Y&B\ar[l]_{s}}\] in which the pair $(k,s)$ is
jointly epimorphic and there exists a (necessarily unique) morphism $p\colon{Y\to B}$ with $ps=1_B$ and $pk=0$.
\end{definition}

Similarly we can define a left patch (by requiring the existence of a morphism $q\colon{Y\to X}$ with $qk=1_X$ and
$qs=0$) but, since here we are going to deal only with right patches we will call them just patches.

Two examples that illustrate the notion can be obtained as follows. Let $\B$ be a pointed category with kernels and
pushouts.
\begin{enumerate}
\item Every coproduct diagram in $\B$ is a patch
\[\xymatrix{X\ar[r]^(.4){\iota_X}&X+B\ar@<.5ex>@{-->}[r]^(.6){[0,1]}&B\ar@<.5ex>[l]^(.4){\iota_B}.}\]

\item If we denote by $k_0\colon{B\flat X\to X+B}$ the kernel of $[0,1]\colon{X+B\to B}$ and let
$\eta_X\colon{X\to B\flat X}$ be such that $k_0\eta_X=\iota_X$, then every morphism
$\xi\colon{B\flat X\to X}$ satisfying the condition $\xi\eta_X=1_X$ induces,
by taking the pushout of $\xi$ and $k_0$, a patch in $\B$  as illustrated by
\[\xymatrix{X\flat B\ar[r]^{k_0}\ar@<-.5ex>[d]_{\xi}&X+B\ar[d]^{\iota_2}\ar@<.5ex>[r]^(.6)
{[0,1]}&B\ar@<.5ex>[l]^(.4){\iota_B}\\X\ar@<-.5ex>[u]_{\eta_X}\ar[r]^{\iota_1}& Q \ar@<.5ex>@{-->}[r]^(.5){p}
&B\ar@<.5ex>[l]^{\iota_2\iota_B}\ar@{=}[u]}\]
The needed morphism $p$ is uniquely determined by $p\iota_2=[0,1]$ and $p\iota_1=0$. Moreover, since
$(Q,\iota_1,\iota_2)$ is a pushout diagram, we have that $(\iota_1,\iota_2)$ is a jointly epimorphic pair.
In order to prove that the pair $(\iota_1,\iota_2\iota_B)$ is jointly epimorphic we  observe that
$\iota_1=\iota_2\iota_X$, indeed
\[\iota_1=\iota_1\xi\eta_X=\iota_2k_0\eta_X=\iota_2\iota_X,\]
and from here it follows that  $(X,Q,B,\iota_1,\iota_2\iota_B,p)$ is a patch.
\end{enumerate}


It will be relevant for us to  differentiate the patches that are exact and the patches that are stable under pullback,
according to the following definitions:

\begin{definition}\label{def:exactpatch}
A patch $(X,Y,B,k,s,p)$ in $\B$ is said to be an \textbf{exact patch} if the morphism $k\colon{X\to Y}$
is the kernel of
the morphism $p\colon{Y\to B}$.
\end{definition}

The morphisms $\xi\colon{B\flat X\to X}$, in the second example above, that induce an exact patch are precisely the
strict actions in the sense of \cite{MF+S:12}, see also \cite{Hartl}. Moreover, in the category of pointed sets and in
the category of abelian groups every coproduct diagram is an exact patch. Indeed, in both cases, we have
that $\iota_X$ is the kernel of $[0,1]$.

\begin{definition}\label{def:stablepatch}
A patch $(X,Y,B,k,s,p)$ in $\B$ is said to be a \textbf{stable patch} if for every $h\colon{Z\to B}$, the pullback of
$p$ along $h$ exists in $\B$, and the induced cospan
\[\xymatrix{X\ar[r]^(.4){\langle k,0\rangle} & Y\times_B Z & Y\ar[l]_(.35){\langle sh,1_Y\rangle}}\]
is a patch in $\B$.
\end{definition}

In the category of abelian groups every coproduct diagram is a stable patch. This is not true in the category of pointed
sets. Indeed any cospan
\[\xymatrix{X\ar[rr]^(.3){\langle \iota_X,0\rangle}&&(X+B)\times_{B} Y&&Y\ar[ll]_(.3){\langle \iota_B h,1_Y\rangle}}\]
which is obtained by taking the pullback of the morphism $$[0,1]\colon{X+B\to B}$$ along a given morphism $h\colon{Y\to B}$
is a patch if and only if the kernel of $h$ is trivial.

Let $I\colon{\A\to\B}$ be a functor. We recall that:
\begin{enumerate}
\item[] A morphism $\alpha\colon{E\to A}$ in $\A$ is \textbf{cartesian} (or $I$-cartesian) if for every $g\colon{W\to A}$ in
$\A$ and every $h\colon{I(W)\to I(E)}$ in $\B$, with $I(\alpha)h=I(g)$, there exists a unique $u\colon{W\to E}$ in $\A$
such that $\alpha u=g$ and $I(u)=h$.
\item[] When every morphism in $\B$ can be lifted to an $I$-cartesian morphism in $\A$ we say that the functor $I$ is a \textbf{fibration}. More specifically, the functor $I$ is a fibration if for every $A$ in $\A$ and $p\colon{Y\to IA}$ in $\B$ there exists a cartesian
morphism (called the cartesian lifting of $p$ along $A$), $\alpha\colon{E\to A}$, with $I(\alpha)=p$.
\end{enumerate}


From now on we consider, other than the functor $I$ another functor $J$.
Let $(I,J)$ be an ordered pair of functors $I,J\colon{\A\to\B}$, which will be displayed either horizontally or
vertically as
\[\xymatrix{\A\ar@<-.5ex>[r]_{J}\ar@<.5ex>[r]^{I}&\B}\quad \text{ or }\quad \xymatrix{\A\ar@<-.5ex>[d]_{J}
\ar@<.5ex>[d]^{I}\\\B.}\]
In this context we consider a special class of morphisms in $\A$ that we call \emph{organic} (due to the fact
that their components under $I$ and under $J$  form an exact patch). Note that on the vertical display the functor $J$ appears on the left while the functor  $I$ appears on the right, although the ordered pair of functors is $(I,J)$.

\begin{definition}\label{def:organic}
A morphism $f\colon{A\to E}$ in $\A$ is said to be a \textbf{organic} morphism (or $(I,J)$-organic) if $J(E)\cong I(E)$ and the
cospan \[\xymatrix{JA\ar[r]^(.4){J(f)}&JE\cong IE&IA\ar[l]_(.3){I(f)}}\] is an exact (right) patch in $\B$.
\end{definition}


Finally, we complete the setting by introducing another ingredient --- the Whitehead sequence --- and
the definition of $L$-condition and of action-system.

Let $(I,G,J)$ be an ordered triple of functors, displayed as
\begin{equation}\label{eq:IGJ}
\vcenter{\xymatrix@=3em{\A \ar@<1.5ex>[r]^-{I} \ar@<-1.5ex>[r]_-{J} & \B \ar[l]|-{G} }}
\end{equation}
and such that $IG=1_{\B}=JG$.

\begin{definition}\label{def:Whitehead}
A  \textbf{Whitehead sequence} is a triple $(A,u,v)$ where $A$ is an object in $\A$, while $u$ and $v$ are morphisms
in $\A$, of the form
\begin{equation}\label{eq:Whitehead}
\xymatrix{GJ(A)\ar[r]^(.6){v}&A\ar[r]^(.4){u}&GI(A),}
\end{equation}
satisfying the following conditions
\begin{eqnarray}
I(u)=1_{IA}\\
J(v)=1_{JA}\\
I(v)=J(u).
\end{eqnarray}
\end{definition}


\begin{definition}\label{def:Loday}
We say that the \textbf{L-condition} holds for the triple of functors $(I,G,J)$ when for every diagram of solid arrows
\[\xymatrix{GJE\ar@{-->}[r]^{g'}\ar[rd]_{g}&E\ar@{-->}[r]^{f'}\ar@<-.5ex>[d]_{\alpha}&GIE\\&A\ar@<-.5ex>[u]_{\beta}
\ar[ru]_{f}}\]
with $I(\beta)=I(f)$, $J(\alpha)=J(g)$, $I(\alpha)J(f)=I(g)J(\beta)$ and $\alpha\beta=1_A$, if $\alpha$ is cartesian
and $f$ is organic then there exists a unique Whitehead sequence $(f',g')$ such that $\alpha g'=g$ and $f'\beta=f$.
\end{definition}

\begin{definition}\label{def:actionsystem}
A triple of functors $(I,G,J)$ is called an \textbf{action-system} of $\A$ over $\B$ when:
\begin{enumerate}
\item the functor $I$ is a fibration and $J(\alpha)$ is an isomorphism whenever $\alpha$ is a cartesian morphism;
\item for every $A$ in $\A$ there exists an object $Y\in B$ and a morphism $$f\colon{A\to G(Y)}$$ such that
$f$ is organic and, moreover, it is universal from $A$ to $G$;
\item the L-condition holds.
\end{enumerate}
\end{definition}

The three main examples that have motivated these definitions that is
(a) groups,
(b) abelian groups and
(c) pointed sets, are presented in some detail in the following section. It is expected that,
due to their generality, these definitions will be applicable in a wide variety of cases, allowing, in particular,
the study of internal categories and internal groupoids via Whitehead sequences in general contexts.

Some immediate consequences of the definitions are the following.

\begin{proposition}\label{prop:basicprops}
Let $(I,G,J)$ be an action-system of $\A$ over $\B$. Then
\begin{enumerate}
\item[(i)] the functor $G$ has a left adjoint;

\item[(ii)] there exists a unique natural transformation $\pi\colon{1_{\A}\to GI}$ such that for every object
$A$ in $\A$, $I(\pi_A)=1_{IA}$ and $J(\pi_A)=0$;

\item[(iii)] there exists a functor $\A\to\Pt(\B)$;

\item[(iv)] for every Whitehead sequence $(A,u,v)$ there exists, up to isomorphism, a unique diagram
in $\A$
\[\xymatrix{E\ar[r]^(.4){\mu}\ar@<-.5ex>[d]_{\alpha}&GFA\\A\ar@<-.5ex>[u]_{\beta}\ar[ru]_{\eta_A}}\]
in which $\eta_A$ is the universal arrow from $A$ to the functor $G$, $\alpha$ is a cartesian morphism,
$IE\cong FA$, and such that
\begin{eqnarray*}
GI(\alpha)\eta_A&=&u\\
\alpha\beta&=&1_A\\
I(\beta)&=&I(\eta_A)\\
\mu\beta&=&\eta_A\\
I(\mu)&=&1_{FA};
\end{eqnarray*}
\item[(v)] every Whitehead sequence $(A,u,v)$ induces another Whitehead sequence, say $(E,\mu,\nu)$, with the
property that there exists a cartesian split epimorphism $$\alpha\colon{E\to A}$$ (with a section $\beta$) such
that $\mu\beta=\eta_A$ and $\alpha\nu=vGJ(\alpha)$;
\item[(vi)] every Whitehead sequence $(A,u,v)$ induces an infinite sequence of cartesian split epimorphisms
\[
\xymatrix{\cdots &
A_3 \ar[r]^{\alpha_3} & A_2 \ar[r]^{\alpha_2} & A_1 \ar[r]^{\alpha_1} & A_0=A
}
\] which is uniquely determined by $GI(\alpha_1)\eta_{A}=u$ and if  $\beta_i$ is the section of $\alpha_i$,
for every $i=1,2,\ldots$, by the equations
\begin{eqnarray*}
\alpha_i\beta_i&=&1_{A_{i-1}}\\
I(\beta_i)&=&I(\eta_{A_{i-1}})\\
GI(\alpha_{i+1})\eta_{A_i}\beta_{i}&=&\eta_{A_{i-1}}\\
I(\alpha_{i+1})I(\eta_{A_i})&=&1_{IA_i}.
\end{eqnarray*}
\end{enumerate}

\end{proposition}
\begin{proof} $\quad$
\begin{enumerate}
\item[(i)] Since, by Definition \ref{def:actionsystem}(2), for every object $A$ in $\A$, there exists an object
$FA$ in $\B$ and an arrow $\eta_A\colon{A\to GFA}$ which, in particular, is universal from $A$ to the functor $G$,
it follows directly from Theorem~2(ii), on page 83 of \cite{ML}, that $G$ is (the right) part of an adjuntion
$(F,G,\eta,\varepsilon)$.
\item[(ii)] Using the previous adjunction $(F,G,\eta,\varepsilon)$, we observe that the existence of a
morphism $\pi_A\colon{A\to GIA}$ such that $I(\pi_A)=1_{IA}$ and $J(\pi_A)=0$, is equivalent to the existence of a
morphism \[\xymatrix{FA\ar[rr]^{\varepsilon_{IA}F(\pi_A)}&&IA}\]such that $\varepsilon_{IA}F(\pi_A)I(\eta_A)=1_{IA}$
and $\varepsilon_{IA}F(\pi_A)J(\eta_A)=0$. The assumption (see Definitions \ref{def:actionsystem}(2), \ref{def:organic}
and \ref{def:patch}) that $\eta_A\colon{A\to GFA}$ is a patch guarantees the existence, as well as the uniqueness,
of $\varepsilon_{IA}F(\pi_A)$ and hence of $\pi_A$. The naturality of $\pi$ follows from the naturality of $\eta$
and $\varepsilon$. Further details on this construction can be found in \cite{BeppeNMF}.
\item[(iii)] Using again the adjunction $(F,G,\eta,\varepsilon)$ and the natural transformation
$$\pi\colon{1_{\A}\to GI},$$
from the two items above, we observe that to every $A$ in $\A$ we can associate the split extension
\[
\xymatrix@=3em{JA  \ar[r]^-{J(\eta_A)} &  FA \ar@<0.5ex>[r]^-{\varepsilon_{IA}F(\pi_A)} & IA \ar@<.5ex>[l]^-{I(\eta_A)}.}
\]
Further details about this construction can be found in \cite{BeppeNMF}.
\item[(iv)] Let $(A,u,v)$ be a Whitehead sequence. We will first show how to obtain the morphisms $\alpha$, $\beta$
and $\mu$ and then show that they are uniquely determined by the properties required.
The morphism $\alpha\colon{E\to A}$ is the cartesian lifting of the morphism $\varepsilon_{IA}F(u)\colon{FA\to IA}$,
which exists because the functor $I$ is a fibration,
and it is such that $IE=FA$ and $I(\alpha)=\varepsilon_{IA}F(u)$ or equivalently, via the adjunction,
$GI(\alpha)\eta_A=u$. The morphism $\beta\colon{A\to E}$ is obtained as the unique morphism such that
$\alpha\beta=1_A$ and $I(\beta)=I(\eta_A)$ which exists because $\alpha$ is cartesian and
$I(\alpha)I(\eta_A)=1_A$ (this is a consequence of $I(u)=1_{IA}$). The morphism $\mu$ is obtained by
applying  the L-condition (Definitions \ref{def:actionsystem}(3) and \ref{def:Loday}) to the   diagram
\[\xymatrix{GJE\ar@{-->}[r]^{\nu}\ar@<0ex>[d]_{GJ(\alpha)}&E\ar@{-->}[r]^(.4){\mu}\ar@<-.5ex>[d]_{\alpha}&GFA\\GJA
\ar[r]^{v}&A\ar@<-.5ex>[u]_{\beta}\ar[ur]_{\eta_A},}\]
which satisfies the needed conditions to guarantee the existence of $\mu$ such that $\mu\beta=\eta_A$ and
$I(\mu)=1_{FA}$. It remains to show that $\mu$ is uniquely determined by this two conditions.
The morphism $\nu$ is uniquely determined because $\alpha$ is cartesian and hence, by the uniqueness property
in the L-condition, we conclude that also $\mu$ is uniquely determined.
\item[(v)] It follows from the Whitehead sequence constructed in the previous item.
\item[(vi)] Having a Whitehead sequence $(A,u,v)$ and using the construction on the previous item we obtain
$\alpha_1$ and $\beta_1$ together with a new Whitehead sequence  $(A_1,\mu_1,\nu_1)$. This gives us the first element of the infinite sequence. We can continue the sequence by replacing $(A_0,u,v)$ with $(A_1,\mu_1,\nu_1)$ and
thus successively iterate in order to obtain $(A_{n},\mu_{n},\nu_{n})$ for an arbitrary $n$. At each level
$i=1,2,\ldots$, the morphism $\beta_i$ is completely determined by $\alpha_i\beta_i=1_{A_{i-1}}$ and
$I(\beta_i)=I(\eta_{A_{i-1}})$. In the same way the morphism $\alpha$, being a cartesian morphism,
is completely determined by $GI(\alpha_{i+1})\eta_{{A_i}}=\mu_i$. But, since $\mu_i$ itself is determined
by $\mu_i\beta_i=\eta_{A_{i-1}}$ and $I(\mu_i)=1_{I(A_i)}$, the two equations
\begin{eqnarray*}
GI(\alpha_{i+1})\eta_{A_i}\beta_{i}&=&\eta_{A_{i-1}}\\
I(\alpha_{i+1})I(\eta_{A_i})&=&1_{IA_i}
\end{eqnarray*}
uniquely determine $\alpha_i$.
\end{enumerate}
\end{proof}

We are now going to see the main examples of action-systems that motivated the definitions above.

\section{Pointed sets, groups and abelian groups}\label{sec:examples}

Let $\B$ be the category of abelian groups and  $\A$ the category  $\B\times \B$ with $I$ the second projection,
$J$ the first projection and $G$ the diagonal functor. The triple of functors $(I,G,J)$ is an action system of
$\A$ over $\B$. As we sill see, the same is true for the category of pointed sets and, more generally, in any
category $\B$ provided it is pointed,  has binary coproducts and such that, for every two objects  $X$ and $B$,
the morphism $\iota_X\colon{X\to X+B}$ is the kernel of $[0,1]\colon{X+B\to B}$.

Some simple observations presented next to support our claims are to be compared
with the respective items from Definition \ref{def:actionsystem} of an action-system:
\begin{enumerate}
\item the functor $I$ is a fibration and $J(\alpha)$ is an isomorphism if and only if $\alpha$ is cartesian;
\item every $A=(X,B)$ in $\A$ has an object $X+B$ in   $\B$  and an arrow
\[(\iota_X,\iota_B)\colon{(X,B)\to (X+B,X+B)}\]
which is organic and universal;
\item to check that the L-condition holds we have to consider a diagram of solid arrows of the form
\[\xymatrix{(X,X)\ar@{-->}[r]^{(1,k)}\ar[rd]_{(1,h)}&(X,Y)\ar@{-->}[r]^{(k,1)}\ar@<-.5ex>[d]_{\alpha}&(Y,Y)\\&(X,B)\ar@<-.5ex>[u]_{\beta}\ar[ru]_{(k,s)}}\]
where we assume that $\alpha$ is cartesian, which means that, up to an isomorphism, we can write it as
$\alpha=(1_X,\alpha_2)$, and $(k,s)$ is an exact patch, which means that $(k,s)$ is a jointly epimorphic cospan
and there exists $p\colon{Y\to B}$ with $ps=1_B$ and $k$ the kernel of $p$;
the remaining assumptions only give $\alpha_2k=h$ and we easily confirm the existence of a unique Whitehead sequence
(dashed arrows) satisfying the desired equations.

Note that a Whitehead sequence $(A,u,v)$, in this case, is completely determined by either $I(v)$ or $J(u)$.
In other words, it is completely determined by a morphism $h\colon{X\to B}$ and it is of the form
\[\xymatrix{(X,X)\ar[r]^{(1,h)}&(X,B)\ar[r]^{(h,1)}&(B,B).}\]

\end{enumerate}

Another example, in fact the main example since it was the main motivation of this work, is the case where $\B$ is
the category of groups and $\A$ is the category of group actions.

Classically, an action of a group $B$ on a set $X$ is a map $\xi \colon{B \times X \to X}$ assigning to every pair
$(b,x)$ in $B \times X$ an element $b\cdot x$ in $X$  such that $1\cdot x=x$ and $(bb')\cdot x=b\cdot(b'\cdot x)$.
Equivalently it may be presented as a group homomorphism
\[
\phi \colon{ B \to \Aut(X)}
\]
from the group $B$ to the automorphism group of $X$.
Another approach consists on considering the group $B$ as a one object groupoid and an action as a functor
\[
B \to \Set
\]
assigning the set $X$ to the (only) object of the groupoid $B$ and
an automorphism of $X$ to each morphism in the groupoid $B$ (that is to each element of the group $B$).
A convenient notation that illustrates this situation is the following one.
\begin{eqnarray*}
\xymatrix{B\ar[r]^{X}&\Set}\\
\xymatrix{\circ\ar[d]^{b}="b" & X_{\circ}\ar[d]_{X_b}="Xb" \\ \circ & X_{\circ} \ar@{|->}"b";"Xb" }
\end{eqnarray*}
In this language the conditions above are written as
\[X_1=1_{X_{\circ}} \] and \[ X_{b'} X_{b} =X_{b'b}. \]

Again, in classical terms, a morphism between actions is a pair $(f,g)$
\[
\xymatrix{
(X,\xi,B) \ar[d]^{(f,g)} \\
(X',\xi',B')
}
\]
in which $g\colon{B\to B'}$ is a group homomorphism while $f\colon{X\to X'}$ is a map, such that
\[
f(b\cdot x)=g(b)\cdot f(x).
\]
Equivalently, it may be considered as a morphism in a (super) comma-category
\[
\xymatrix{
B \ar[d]_{g} \ar[rr]^{X} & \ar@{}[d]|{\Downarrow f} & \Set \\
B' \ar[rr]^{X'} & & \Set
}
\]
where $f:X \to X' g$ is a natural transformation.

It is clear that instead of the category $\Set$ we can consider other categories, obtaining there
an appropriate notion of group action. In particular, if we consider the category $\Grp$ of groups
we  obtain the category of
group actions on groups.

Let us consider now the case of an action-system where $\B$ is the category of groups and $\A$ is the
category of group actions on groups. An object $A$ in $\A$ is a pair $(X,B)$ in which $B$ is a group
(considered as a one object groupoid) and $X\colon{B\to\Grp}$ is a functor. The morphisms are the pairs
$(f,g)$ with $g\colon{B\to B'}$ a group homomorphism and $f\colon{X\to X'g}$ a natural transformation.

In this case $I$ is the second projection,  $J$ is the first projection
(in the sense that $J(X,B)=X_{\circ}$) and, for every group $B$, $G(B)=(\bar{B},B)$ where
$\bar{B}\colon{B\to \Grp}$ corresponds to the action by conjugation of $B$ onto itself,
that is $\bar{B}_\circ=B$ and $\bar{B}_b(b')=bb'b^{-1}$.

It follows that $(I,G,J)$ is an action-system of $\A$ over $\B$ in which the Whitehead sequences are precisely
the crossed-modules of groups.
Indeed it is not difficult to check that a Whitehead sequence is determined by a pair $(A,h)$ where $A$ is an
object in $\A$, $h:JA \to IA$ is a morphism in $\B$, and there exist two morphisms $u$ and $v$
 \[
\vcenter{\xymatrix{GJA \ar[r]^-{v} & A \ar[r]^-{u} & GIA }}
\]
such that
\[
I(u)=1_{IA},\ J(v)=1_{JA}, \  I(v)=J(u)=h.
\]
In other words a Whitehead sequence becomes a property on the object $A$ and the morphism $h$ which is equivalent to
the two well-known conditions for a crossed module, namely
\begin{eqnarray}
h(b\cdot x)=bh(x)b^{-1} \label{xmod-1}\\
h(x)\cdot x' =x+x'-x \label{xmod-2}
\end{eqnarray}
in which we write $X=JA$ additively, $B=IA$ multiplicatively and denote by $b\cdot x=X_b(x)$ the result of
the action of the element $b$ in $B$ on the element $x$ in $X$. Condition  (\ref{xmod-1}) is equivalent to the
existence of $u$, while condition (\ref{xmod-2}) is equivalent to the existence of $v$.

The functor $I$ is a fibration: the cartesian lifting of a morphism $g\colon{B'\to B}$ in $\B$ along an action $(X,B)$
in $\A$ is given by
\[
\xymatrix{
B' \ar[d]_{g} \ar[rr]^{Xg} & \ar@{}[d]|{\Downarrow 1} & \Grp \\
B \ar[rr]^{X} & & \Grp
}
\]
where $1$ denotes the identity natural transformation of the functor $Xg$. If $\alpha$ is a cartesian morphism
in $\A$ then $J(\alpha)$ is an isomorphism in $\B$. To each action $(X,B)$ in $\A$ we can associate
the semidirect product diagram
\[\xymatrix{X_{\circ}\ar[r]^(.35){k}&F(X,B)&B\ar[l]_(.35){s}}\]
in which $F(X,B)=X_{\circ}\rtimes B$ is the set of pairs $(x,b)\in X_{\circ}\times B$ with the operation
\[(x,b)+(x',b')=(x+X_b(x'),bb')\] and $k$, $s$ are the canonical injections. This diagram is an exact patch
and, moreover, the pair $(k,s)$ can be seen as a universal arrow $$(k,s)\colon{(X,B)\to GF(X,B)}.$$

In order to conclude that the triple $(I,G,J)$ is an action-system of $\A$ over $\B$ it remains to analyse
the L-condition.
In this case it simplifies to a diagram in $\A$ as the one displayed below
\[\xymatrix{(\bar{X_{\circ}},X_{\circ})\ar@{-->}[r]^{(1,k)}\ar[rd]_{(1,h)}&(X\alpha_2,Y)\ar@{-->}[r]^{(k,1)}
\ar@<-.5ex>[d]_(.3){(1_{X\alpha_2},\alpha_2)}&(\bar{Y},Y)\\&(X,B)\ar@<-.5ex>[u]_(.6){(1_X,s)}\ar[ru]_{(k,s)}}\]
in which $\alpha_2s=1_B$ and $\alpha_2k=h$. This diagram comes from assuming that $\alpha=(1_{X\alpha_2},\alpha_2)$
is a cartesian morphism and that all the conditions in the statement of the L-condition are satisfied.
The extra piece of information is the assumption that $f=(k,s)$ is a organic morphism. From this we have to show that
$(1,k)$ and $(k,1)$ are morphisms in $\A$. The fact that $(1,h)$ is a morphism implies that (in fact it is equivalent to)
$Xh$ being equal to the conjugation action on $X_{\circ}$, or in other words $Xh=\bar{X}_{\circ}$. From here we can
conclude
that $(1,k)$ is a morphism since we have $X\alpha_2k=Xh=\bar{X}_{\circ}$.

The requirement that $(k,1)$ is a morphism in $\A$ is equivalent to the requirement that
\[k(\alpha_2(y)\cdot x')=y+k(x')-y\] holds for all $x'\in X_{\circ}$ and all $y\in Y$
(note that we write $X_{\alpha_2(y)}(x')$ as $\alpha_2(y)\cdot x'$ in order to simplify notation).
To prove this condition we now make use of the assumption that the morphism $(k,s)$ is a organic morphism,
which means that the cospan
\[\xymatrix{X_{\circ}\ar[r]^{k}&Y&B\ar[l]_{s}}\]
is an exact patch and hence, every element $y\in Y$ can be written in a unique way as $y=(x,b)$
with $x\in X_{\circ}$ and $b\in B$ and, moreover, $\alpha_2(y)=h(x)+b$. It is now an easy calculation
to verify the desired condition since we have $h(x)\cdot x'=x+x'-x$ because $Xh=\bar{X}_{\circ}$.

\section{A simplicial construction}\label{sec:simp}

In this section we introduce a simplicial construction which will be used in the proof of the main result.
We construct a simplicial object in a category $\B$ from a sequence of cartesian split epimorphisms
in a category $\A$, which is equipped with a \emph{realization} functor into the category of points in $\B$.

Let $\A$ and $\B$ be two categories and suppose that it is given a functor
\[
\A \to \Pt(\B)
\]
from the category $\A$ into the category of points (i.e. split epimorphisms) in $\B$. 
We call such functor a realization functor since it allows to consider (or realise) an object in $\A$ 
as a split epimorphism in $\B$. Giving such a functor is to give an ordered pair of functors
\[
F,I:\A \To \Pt(B)
\]
 (we think of $F$ as the domain functor and of $I$ as the  codomain functor) together with two natural transformations
\[
\pi:F \To I \;\; \text{and} \;\; \iota:I\to F
\]
which are related by the following condition
\[
\pi\iota=1_I.
\]
With this data, $(F,I,\pi,\iota)$, we are able to associate to every $A$ in $\A$ a split epimorphism in $\B$ of the form
\[
\xymatrix{FA \ar@<0.5ex>[r]^{\pi_A} & IA \ar@<0.5ex>[l]^{\iota_A}.}
\]

In the proof of the following proposition we explain how to construct a simplicial object in the category $\B$, 
using the canonical split epimorphisms associated to each object $A$ in $\A$, together with a sequence of 
cartesian split epimorphisms in $\A$.

\begin{proposition}\label{prop:simplicialconstruction}
Let $(F,I,\pi,\iota)\colon{\A\to\Pt(\B)}$ be a functor from $\A$ into the category of split epimorphisms in $\B$. 
Suppose that for every split epimorphism in $\A$,
\[
\xymatrix{E \ar@<0.5ex>[r]^{\alpha} & A \ar@<0.5ex>[l]^{\beta},}
\]
if $\alpha$ is $I$-cartesian then the pair $(F\beta,\iota_E)$ is jointly epimorphic. Then, every sequence of 
split epimorphisms in $\A$ of the form
\begin{equation}\label{eq:splitepiseqinA}
\xymatrix{
... & A_n
\ar@<0.5ex>[r]^{\alpha_n}
& A_{n-1}
\ar@<0.5ex>[l]^{\beta_n}
& ... & A_2
\ar@<0.5ex>[r]^{\alpha_2}
& A_1
\ar@<0.5ex>[r]^{\alpha_1}
\ar@<0.5ex>[l]^{\beta_2}
& A_0
\ar@<0.5ex>[l]^{\beta_1}
}
\end{equation}
in which $\alpha_n$ is cartesian for all $n$, and
\begin{eqnarray}
IA_n & = & FA_{n-1} \nonumber\\
I(\alpha_n) \iota_{n-1} & = & 1_{IA_{n-1}} \label{eq:splitepiseqinAconditions}\\
I(\alpha_{n+1}) F(\beta_{n}) & = & 1_{IA_{n}},\nonumber
\end{eqnarray}
induces a simplicial object in $\B$.
\end{proposition}
 Note that we denote $\pi_{A_n}$ and $\iota_{A_n}$ by $\pi_n$ and $\iota_n$ and omit some parenthesis, 
 so that $I(A)$ becomes $IA$.
\begin{proof}
The simplicial object has the following form
\[
\xymatrix@=3em{
...& IA_3=FA_2
\ar@<6ex>[r]^-{\pi_2}
\ar@<0ex>[r]^-{I(\alpha_3)}
\ar@<-6ex>[r]^-{F(\alpha_2)}
\ar@<-12ex>[r]^-{F^2(\alpha_1)}
& IA_2=FA_1
\ar@<6ex>[r]^-{\pi_1}
\ar@<0ex>[r]^-{I(\alpha_2)}
\ar@<-6ex>[r]^-{F(\alpha_1)}
\ar@<-3ex>[l]_-{\iota_2}
\ar@<3ex>[l]_-{F(\beta_2)}
\ar@<9ex>[l]_-{F^2(\beta_1)}
& IA_1=FA_0
\ar@<6ex>[r]^-{\pi_0}
\ar@<0ex>[r]^-{I(\alpha_1)}
\ar@<-3ex>[l]_-{\iota_1}
\ar@<3ex>[l]_-{F(\beta_1)}
& IA_0
\ar@<-3ex>[l]_-{\iota_0}
}
\]
\[
\xymatrix@=5em{
\cdots & IA_{n+1}=FA_n
\ar@<6ex>[r]^-{\pi_n}
\ar@<0ex>[r]^-{I(\alpha_{n+1})}
\ar@<-6ex>[r]^-{F(\alpha_n)}
\ar@<-12ex>[r]^-{F^2(\alpha_{n-1})}
\ar@<-15ex>@{}[r]^-{\vdots}
\ar@<-21ex>[r]^-{F^i(\alpha_{n-i+1})}
\ar@<-24ex>@{}[r]^-{\vdots}
\ar@<-30ex>[r]^-{F^n(\alpha_{1})}
& IA_n=FA_{n-1}
\ar@<-3ex>[l]_-{\iota_n}
\ar@<3ex>[l]_-{F(\beta_n)}
\ar@<9ex>[l]_-{F^2(\beta_{n-1})}
\ar@<18ex>[l]_-{F^i(\beta_{n-i+1})}
\ar@<27ex>[l]_-{F^n(\beta_{1})}
& \cdots}
\]
in which $F^2(\alpha_1)=F(F(\alpha_1)^*)$ with $F(\alpha_1)^*$ the unique morphism in $\A$ such that 
$\alpha_1 F(\alpha_1)^*=\alpha_1\alpha_2$ and $I(F(\alpha_1)^*)=F(\alpha_1)$, as illustrated in the 
following picture
\[
\xymatrix@=3em{
& A_1 \ar[d]^{\alpha_1} & & IA_1 \ar[d]^{I\alpha_1} \\
A_2
\ar@{-->}[ru]^{F(\alpha_1)^*}
\ar[r]_{\alpha_1\alpha_2}
& A_0
& IA_2
\ar[ru]^{F(\alpha_1)}
\ar[r]_{I(\alpha_1\alpha_2)}
& IA_0,
}
\]
that exists because $\alpha_1$ is $I$-cartesian and the triangle on the right is commutative (see equation (\ref{eq-for-associativity}) below). Similarly, $F^2(\beta_1)=F(F(\beta_1)^*)$ with $F(\beta_1)^*$ the unique morphism in $\A$ such that  $\alpha_2 F(\beta_1)^*=1_{A_1}$ and $I(F(\beta_1)^*)=F(\beta_1)$, as displayed in the following picture
\[
\xymatrix@=3em{
& A_2 \ar[d]^{\alpha_2} & & IA_2 \ar[d]^{I\alpha_2} \\
A_1
\ar@{-->}[ru]^{F(\beta_1)^*}
\ar@{=}[r]_{}
& A_1
& IA_1
\ar[ru]^{F(\beta_1)}
\ar@{=}[r]_{}
& IA_1.
}
\]

In a similar fashion we can obtain $F^i(\alpha_{n-i+1})$ and $F^i(\beta_{n-i+1})$ for all $i$ up to $n$.
The details are omitted since we will not work with $n$ greater than 2.

The necessary equations for the construction of $F^i(\alpha_{n-i+1})$ are satisfied because the pair
\[ (F(\beta_n), \iota_n)   \]
is jointly epimorphic for all $n$. Indeed, for example, the construction of $F^2(\alpha_{n-1})$ depends on the equation
\begin{equation}\label{eq-for-associativity}
I(\alpha_n)F(\alpha_n)=I(\alpha_n\alpha_{n+1})
\end{equation}
which is true because we have
\[
I(\alpha_n)F(\alpha_n)F(\beta_n)=I(\alpha_n)=I(\alpha_n\alpha_{n+1})F(\beta_n)
\]
and (since $\iota$ is natural)
\[
I(\alpha_n)F(\alpha_n)\iota_n=I(\alpha_n)\iota_{n-1}I(\alpha_n)=I(\alpha_n)=I(\alpha_n\alpha_{n+1})\iota_n.
\]

Using the same technique it is possible to check that all the simplicial equations are satisfied, 
a routine but demanding task.
\end{proof}

Let us now consider a simple example of this simplicial construction.

Let $\B$ be a pointed category with binary coproducts. Take $\A$ to be the category $\B\times \B$ and, 
for every pair $(X,B)$ of objects in $\B$, define
\begin{eqnarray*}
I(X,B)=B\\
F(X,B)=X+B\\
\pi_{X,B}=[0,1]\colon{X+B\to B}\\
\iota_{X,B}=\iota_B\colon{B\to X+B}.
\end{eqnarray*}
In this case the functor $I$ is a fibration and a morphism $\alpha=(\alpha_1,\alpha_2)$ in $\A$ is cartesian 
if and only if $\alpha_1$ is an isomorphism. Moreover, for any split epimorphism
\[
\xymatrix{(E_1,E_2) \ar@<0.5ex>[r]^{(\alpha_1,\alpha_2)} & (A_1,A_2) \ar@<0.5ex>[l]^{(\beta_1,\beta_2)}}
\]
in $\A$, if $\alpha_1$ is an isomorphism then the cospan
\[\xymatrix{A_1+A_2\ar[r]^{\beta_1+\beta_2}&E_1+E_2&E_2\ar[l]_(.4){\iota_{E_2}}}\]
 is jointly epimorphic (observe that $(\beta_1+\beta_2)\iota_{A_1}\alpha_1=\iota_{E_1}$).

Now, in the particular case of abelian groups, a sequence such as the one displayed in (\ref{eq:splitepiseqinA}) 
with $\alpha_n$ cartesian for all $n$ and satisfying equations $(\ref{eq:splitepiseqinAconditions})$ 
is completely determined, up to isomorphism, by a morphism \[h\colon{X\to B}\] and it is of the following form
\begin{eqnarray*}
A_0&=&(X,B)\\
A_1&=&(X,X+B)\\
A_2&=&(X,X+(X+B))\cong(X,2X+B)\\
A_n&=&(X,nX+B)\\
\alpha_1&=&(1_X,[h,1_B])\\
\beta_1&=&(1_X,\iota_B)\\
\alpha_2&=&(1_X,[\iota_X,1_{X+B}])\\
\beta_2&=&(1_X,\iota_{X+B})\\
\alpha_{n+1}&=&(1_X,[\iota_X,1_{nX+B}])\\
\beta_{n+1}&=&(1_X,\iota_{nX+B}).
\end{eqnarray*}
In other words, it is completely determined by the first element of the sequence. This is not true in general 
but it gives a way to generate examples. Going back again to a category $\B$, pointed with binary coproducts, 
we can consider a sequence of the form just describd and, by Proposition \ref{prop:simplicialconstruction}, 
we are able to construct the following simplicial object
\begin{equation}
\xymatrix@=3em{
X+(2X+B)
\ar@<6ex>[r]^-{[0,1]}
\ar@<0ex>[r]^-{[\iota_X,1_{2X+B}]}
\ar@<-6ex>[r]^-{1+[\iota_X,1]}
\ar@<-12ex>[r]^-{1+(1+[h,1])}
&
{\begin{matrix}
X+(X+B)\\\cong\\2X+B
\end{matrix}}
\ar@<6ex>[r]^-{[0,1]}
\ar@<0ex>[r]^-{[\iota_X,1_{X+B}]}
\ar@<-6ex>[r]^-{1+[h,1_B]}
\ar@<-3ex>[l]_-{\iota_{2X+B}}
\ar@<3ex>[l]_-{1+\iota_{X+B}}
\ar@<9ex>[l]_-{1+(1+\iota_B)}
& X+B
\ar@<6ex>[r]^-{[0,1]}
\ar@<0ex>[r]^-{[h,1]}
\ar@<-3ex>[l]_-{\iota_{X+B}}
\ar@<3ex>[l]_-{1+\iota_B}
& B
\ar@<-3ex>[l]_-{\iota_B}
}
\end{equation}
which, for simplicity, we truncated at level 3.

\section{The category of Whitehead sequences}\label{sec:cat}

Let $(I,G,J)$ be a triple of functors  as displayed in  $(\ref{eq:IGJ})$ such that $$IG=1_{\B}=JG.$$ 
We consider the category $\W(\A)$ whose objects are the Whitehead sequences in $\A$ (see Definition \ref{def:Whitehead}). 
A morphism $f\colon{(A,u,v)\to(A',u',v')}$ between two Whitehead sequences is a morphism $f\colon{A\to A'}$ in $\A$ 
such that the two squares below are commutative
\[\xymatrix{GJA\ar[r]^{v}\ar[d]_{GJ(f)}&A\ar[r]^{u}\ar[d]_{f}&IGA\ar[d]^{GI(f)}\\GJA'\ar[r]^{v'}&A'\ar[r]^{u'}&IGA'.}\]

When, moreover, the triple of functors $(I,G,J)$ is an action-system of $\A$ over $\B$ 
(definition \ref{def:actionsystem}) and denoting by $(F,G,\eta,\varepsilon)$ the system in which $F$ is the left 
adjoint of $G$ (Proposition \ref{prop:basicprops}(i)), then we can define a full subcategory of $\W(\A)$, 
denoted by $\W^{*}(\A)$, as follows: a Whitehead sequence $(A,u,v)$ is an object in $\W^{*}(\A)$ if every 
cartesian morphism $\alpha\colon{E\to A}$ on its induced sequence of cartesian morphisms 
(as in Proposition \ref{prop:basicprops}(vi)) has the property that the square\[
\vcenter{\xymatrix{FE \ar[rr]^-{\varepsilon_{IE}F(\pi_E)} \ar[d]_{F\alpha}  & & IE \ar[d]^{I\alpha} \\
FA \ar[rr]^-{\varepsilon_{IA}F(\pi_A)} & & IA}}
\]
is a pullback square. The morphisms $\pi_E\colon{E\to GIE}$ and $\pi_A\colon{A\to GIA}$ are the components of 
the natural transformation that is obtained as in the item (ii) of Proposition \ref{prop:basicprops}. For example, 
in the case of the category of groups, together with the action-system of group actions over it 
(as illustrated in Section \ref{sec:examples}), we have that to each cartesian morphism 
$\alpha\colon{E\to A}$ its associated square in the sense above is of the form
\[\xymatrix{J(E)\ltimes I(E) \ar[r]^(.6){[0,1]}\ar[d]_{J(\alpha)\ltimes I(\alpha)}& I(E)\ar[d]^{I(\alpha)}\\J(A)\ltimes I(A)\ar[r]^(.65){[0,1]}&I(A)}\]
which is always a pullback square. Indeed, it simply follows from the fact that $\alpha$ is cartesian and 
hence $J(\alpha)$ is an isomorphism.

We denote by $\Simp(\B)$ the category of internal simplicial objects in $\B$ and consider the category of internal 
categories in $\B$, $\Cat(\B)$, as a full subcategory of $\Simp(\B)$.

\begin{theorem}\label{thm:1}
Let $(I,G,J)$ be an action-system of $\A$ over $\B$. There is a functor from $\W(\A)$ into $\Simp(\B)$ such that 
its restriction to $\W^{*}(\A)$ factors through $\Cat(\B)$
\[\xymatrix{\W(\A)\ar[r]&\Simp(\B)\\W^{*}(\A)\ar@{^{(}->}[u]\ar@{-->}[r]&\Cat(\B).\ar@{^{(}->}[u]}\]
\end{theorem}
\begin{proof}
Following Proposition \ref{prop:basicprops}, to every Whitehead sequence 
$(A,u,v)$ we can associate an infinite sequence of cartesian split epimorphisms $\alpha_i$, with section $\beta_i$,
\[
\xymatrix{
\cdots &A_3 \ar[r]^{\alpha_3} & A_2 \ar[r]^{\alpha_2} & A_1 \ar[r]^(.4){\alpha_1} & A_0=A
}
\]
such that $I(\alpha_1)=\varepsilon_{IA}F(u)$ and for every $i=1,2,\ldots$
\begin{eqnarray*}
I(\alpha_{i+1})&=&\varepsilon_{IA_i}F(\mu_i)\\
\alpha_i\beta_i&=&1_{A_{i-1}}\\
I(\beta_i)&=&I(\eta_{A_{i-1}}).
\end{eqnarray*}
Here, $(F,G,\eta,\varepsilon)$ is the adjunction as in Proposition \ref{prop:basicprops}(i), and 
$(A_1,\nu_1,\mu_1)$ is the Whitehead sequence obtained (as in item (v) of Proposition \ref{prop:basicprops}) 
from the Whitehead sequence $(A,u,v)$. Similarly 
 we obtain $(A_{i+1},\nu_{i+1},\mu_{i+1})$ from  $(A_i,\nu_i,\mu_i)$ for all $i\in \mathbb{N}$.

It follows that
\begin{eqnarray*}
I(A_i)&\cong& F(A_{i-1})\\
I(\alpha_i)I(\eta_{A_{i-1}})&=&
I(\alpha_i)I(\beta_{i})=1_{IA_{i-1}}\\
I(\alpha_{i+1})F(\beta_i)&=&\varepsilon_{A_i}F(\mu_i)F(\beta_i)=\varepsilon_{IA_i}F(\mu_i\beta_i)\\
&=&\varepsilon_{IA_i}F(\eta_{A_{i-1}})=\varepsilon_{FA_{i-1}}F(\eta_{A_{i-1}})\\
&=&1_{FA_{i-1}}=1_{IA_i}.
\end{eqnarray*}

In order to make use of Proposition \ref{prop:simplicialconstruction} with the sequence of cartesian morphisms as 
constructed above, the natural transformation $\pi_i=\varepsilon_{IA_i}F(\pi_{A_i})$ (with $\pi_{A_i}$ obtained as in 
item (ii) of Proposition \ref{prop:basicprops}) and with $\iota_i=I(\eta_{A_i})$ we have to verify that 
the pair $(F(\beta_i),I(\eta_{A_i}))$ is jointly epimorphic.  This is a consequence of the fact that, 
for each $A\in \A$,  $\eta_{A}$ is a organic morphism (Definition \ref{def:organic}) and hence the cospan 
$(J(\eta_A),I(\eta_A))$ is jointly epimorphic. In particular, this implies that  
$(F(\beta_i),I(\eta_{A_i}))$ is jointly epimorphic because each  $J(\beta_i)$ is an isomorphism 
(since $\alpha$ is cartesian then $J(\alpha)$ is an isomorphism and so also $J(\beta)$ is an isomorphism) 
and we have
\[J(\eta_{A_i})=F(\beta_i)
J(\eta_{A_i-1})J(\beta_i)^{-1}.\]

From here we can construct a simplicial object, in the same way as it was done in the proof of 
Proposition \ref{prop:simplicialconstruction}, which is displayed up to level 3 (to compare it with the 
notion of an internal category we will not need to go further) as follows:
 
\[
\xymatrix@=3em{
FA_2
\ar@<6ex>[r]^-{\pi_2}
\ar@<0ex>[r]^-{I(\alpha_3)}
\ar@<-6ex>[r]^-{F(\alpha_2)}
\ar@<-12ex>[r]^-{F^2(\alpha_1)}
& IA_2=FA_1
\ar@<6ex>[r]^-{\pi_1}
\ar@<0ex>[r]^-{I(\alpha_2)}
\ar@<-6ex>[r]^-{F(\alpha_1)}
\ar@<-3ex>[l]_-{\iota_2}
\ar@<3ex>[l]_-{F(\beta_2)}
\ar@<9ex>[l]_-{F^2(\beta_1)}
& IA_1=FA_0
\ar@<6ex>[r]^-{\pi_0}
\ar@<0ex>[r]^-{I(\alpha_1)}
\ar@<-3ex>[l]_-{\iota_1}
\ar@<3ex>[l]_-{F(\beta_1)}
& IA_0
\ar@<-3ex>[l]_-{\iota_0}
}
\]
%
%
%
%

Again, checking the simplicial conditions is a routine (although a demanding) task.

This shows that we can assign a simplicial object to every Whitehead sequence and, moreover, that this construction 
is functorial. Indeed, if $$f\colon{(A,v,u)\to (A',u',v')}$$ is a morphism between Whitehead sequences then it can be 
lifted to the level of infinite sequences of cartesian split epimorphisms so that it respects the 
simplicial equations. This is possible because the morphisms $\alpha_i$ are cartesian and we will have
\[\xymatrix{A_{i+1}\ar[r]^{\alpha_{i}}\ar[d]_{f_{i}}&A_{i-1}\ar[d]^{f_{i-1}}\\A'_{i}\ar[r]^{\alpha'_{i}}&A'_{i-1}}\]
for all $i\in \mathbb{N}$ with $f_0=f$.

This gives us a functor from $\W(\A)$ into $\Simp(\B)$. In order to be able to compare the simplicial structure defined 
above with the one of an internal category, we now give a diagram with the standard notation for an internal category 
object in $\B$.
An internal category in $\B$ is a diagram of the form
\begin{equation}\label{int-cat}
\xymatrix@=3em{
C_3
\ar@<6ex>[r]^-{q_2}
\ar@<0ex>[r]^-{m_2}
\ar@<-6ex>[r]^-{m_1}
\ar@<-12ex>[r]^-{q_1}
& C_2
\ar@<6ex>[r]^-{p_2}
\ar@<0ex>[r]^-{m}
\ar@<-6ex>[r]^-{p_1}
\ar@<-3ex>[l]_-{i_2}
\ar@<3ex>[l]_-{i_0}
\ar@<9ex>[l]_-{i_1}
& C_1
\ar@<6ex>[r]^-{d}
\ar@<0ex>[r]^-{c}
\ar@<-3ex>[l]_-{e_2}
\ar@<3ex>[l]_-{e_1}
& C_0
\ar@<-3ex>[l]_-{e}
}
\end{equation}
where $C_0$ and $C_1$ are, respectively, the object of objects and the object of morphisms, 
while $d,e,c$ are, respectively, domain, identity, and codomain; $C_2$ is the object of composable pairs, 
obtained by the following pullback (with $p_1$, $p_2$ the canonical projections and $e_1$, $e_2$ the induced inclusions)
\[
\xymatrix{
C_2
\ar@<.5ex>[r]^{p_2}
\ar@<-.5ex>[d]_{p_1}
& C_1
\ar@<.5ex>[l]^{e_2}
\ar@<-.5ex>[d]_{c}
\\ C_1
\ar@<-.5ex>[u]_{e_1}
\ar@<.5ex>[r]^{d}
& C_0
\ar@<-.5ex>[u]_{e}
\ar@<.5ex>[l]^{e}.
}
\]
Similarly, $C_3$ is the object of composable triples, specifically calculated for generalized objects as
\[
C_3=\{((f,g),(h,k)) \mid (f,g),(h,k) \in C_2 , g=hk\}
\]
in other words it is the object in the following pullback diagram, of $m$ along $p_2$
\[
\xymatrix{
C_3
\ar@<.5ex>[r]^{q_2}
\ar@<-.5ex>[d]_{m_1}
& C_2
\ar@<.5ex>[l]^{i_2}
\ar@<-.5ex>[d]_{m}
\\ C_2
\ar@<-.5ex>[u]_{i_1}
\ar@<.5ex>[r]^{p_2}
& C_1
\ar@<-.5ex>[u]_{e_1}
\ar@<.5ex>[l]^{e_2}
.}
\]
Note that $C_3$ can also be given by the following pullback 
\[
\xymatrix{
C_3
\ar@<.5ex>[r]^{q_2}
\ar@<-.5ex>[d]_{q_1}
& C_2
\ar@<.5ex>[l]^{i_2}
\ar@<-.5ex>[d]_{p_1}
\\ C_2
\ar@<-.5ex>[u]_{i_1}
\ar@<.5ex>[r]^{p_2}
& C_1
\ar@<-.5ex>[u]_{e_1}
\ar@<.5ex>[l]^{e_2}
}
\]
which is equivalent, being then $C_3$ 
\[
C_3=\{((f,g),(h,k)) \mid (f,g),(h,k) \in C_2 , g=h\}.
\]
To the reader not familiar with the above notation for internal categories, and in order to easily compare 
it with the more standard simplicial one, it may be helpful to consider the particular case where $C_0=1$ 
and write $m(x,y)=xy$, in this case we have
\begin{eqnarray}
p_2(x,y) & = & y \nonumber\\
p_1(x,y) & = & x \nonumber\\
e_1(x) & = & (x,1) \nonumber\\
e_2(y) & = & (1,y) \nonumber\\
q_2(x,y,z) & = & (y,z) \nonumber\\
q_1(x,y,z) & = & (x,y) \nonumber\\
m_1(x,y,z) & = & (x,yz) \nonumber\\
m_2(x,y,z) & = & (xy,z) \nonumber\\
i_1(x,y) & = & (x,y,1) \nonumber\\
i_2(y,z) & = & (1,y,z) \nonumber\\
i_0(x,z) & = & (x,1,z). \nonumber\\
\end{eqnarray}

Table \ref{table:01} translates the (relevant) simplicial equations into the definition of internal category. 
The first column contains the equation in the context of an internal category; 
the middle column presents the equivalent simplicial equation, obtained by the simplicial construction above; 
the last column gives the corresponding equation in the context of $\A$ and $\W(\A)$ where we can easily see why 
the equation is true: lines 1 to 6, by definition; lines 7, 10 and 11 by naturality; lines 9 and 12, 
see equation (\ref{eq-for-associativity}); it remains to explain line 8 --- it follows from the fact 
that $I(\mu)=1=I(\pi_{A_1})$ and $\pi_0J(\eta_A)=0=J(\pi_{A_1})$, since $\eta_A$ is organic for every $A$ in $\A$.

\begin{table}[ht]
\centering
\begin{tabular}{|c|c|c|c|}
\hline
 & \Cat(\B) & \Simp(\B) & \W(\A)
\\\hline
 1  &  $de=1$ & $\pi_0\iota_0=1$ & $I(\pi_A)=1$
\\\hline
 2  & $ce=1$  & $I(\alpha_1)\iota_0=1$ & $I(\mu)=1$
\\\hline
  3 & $p_2e_2=1$  & $\pi_1\iota_1=1$ &  $I(\pi_{A_1})=1$
\\\hline
 4  &  $me_2=1$   &  $m\iota_1=1$  &  $I(\mu)=1$
\\\hline
 5  &  $me_1=1$   &  $mF(\beta_1)=1$  &  $\mu_1\beta_1=\eta_A$
\\\hline
 6  &  $p_1e_1=1$   &  $F(\alpha_1)F(\beta_1)$  &  $\alpha_1\beta_1=1$
\\\hline
 7  &  $cp_2=dp_1$   &  $I(\alpha_1)\pi_1=\pi_0 F(\alpha)$  &  $\pi_A \alpha=GI(\alpha)\pi_{A_1}$
\\\hline
 8   &  $dp_2=dm$   &  $\pi_0\pi_1=\pi_0 m$  &  $G(\pi_0)\mu_1=G(\pi_0)\pi_{A_1}$
\\\hline
 9   &  $cp_1=cm$   &  $I(\alpha_1) F(\alpha_1)=I(\alpha_1)I(\alpha_2)$  &  $u\a_1=GI(\a_1)\mu_1$
\\\hline
 10   &  $p_2e_1=ed$   &  $\pi_1F(\b_1)=\iota_0\pi_0$  &  $\pi_{A_1}\b_1=GI(\b_1)\pi_A$
\\\hline
 11 &  $p_1e_2=ec$   &  $F(\a_1)\iota_1=\iota_0 I(\a_1)$  &  $F(\a_1)I(\eta_{A_1})=I(\eta_A)I(\a_1)$
\\\hline
 12  &  $mm_1=mm_2$   &  $I(\a_2)F(\a_2)=I(\a_2)I(\a_3)$  &  $\mu_1\a_2=GI(\a_2)\mu_2$
\\\hline
\end{tabular}
\caption{Translation between equations: from the language of internal categories, to simplicial objects, to Whitehead sequences.}\label{table:01}
\end{table}


We now have the following: if the squares
\[
\xymatrix{
FA_1
\ar[r]^{\pi_1}
\ar[d]_{F(\a_1)}
& IA_1
\ar[d]^{I(\a_1)}
\\ FA
\ar[r]^{\pi_0}
& IA} \;\; , \;\;
\xymatrix{
FA_2
\ar[r]^{\pi_2}
\ar[d]_{F(\a_2)}
& IA_2
\ar[d]^{I(\a_2)}
\\ FA_1
\ar[r]^{\pi_1}
& IA_1}
\]
are pullbacks, then the simplicial object constructed above is, in fact, an internal category object in $\B$.
This proves that $\W^{*}(\A)$ factors through $\Cat(\B)$. Indeed an object of $\W(\A)$ is in $\W^{*}(\A)$ as soon as every morphism $\alpha$ in its induced infinite sequence of cartesian split epimorphisms has the property that the square
\[
\xymatrix@=3em{
FE
\ar[r]^{\varepsilon_{IE} F(\pi_E)}
\ar[d]_{F(\a)}
& IE
\ar[d]^{I(\a)}
\\ FA
\ar[r]^{\varepsilon_{IA} F(\pi_A)}
& IA}
\]
is a pullback.
\end{proof}

\section{Groupoids and Whitehead sequences}\label{sec:main result}

We are now interested in the case when there is an equivalence between the category of Whitehead sequences in 
$\A$ and the category of internal groupoids in  $\B$, as it is the case, for example, for the category of groups 
and group actions.

\begin{theorem}\label{thm:2}
Let $(I,G,J)$ be an action-system of $\A$ over $\B$. If the pair of functors $(I,J)$ is jointly conservative and there is an equivalence of categories \[
\A \sim \Pt(\B),
\]
compatible with the system $(I,G,J)$, then there is an equivalence of categories
\[
\W^{*}(\A) \sim \Gpd(\B).
\]
\end{theorem}
\begin{proof}
Suppose we have an equivalence of categories
\[
\A \stackrel{\sim}{\longrightarrow} \Pt(\B)
\]
which is compatible with the action-system, that is, an object $A$ in $\A$ is realized as a point of the form
\[
\xymatrix@=3em{
FA \ar@<.5ex>[r]^{\varepsilon_{IA}F(\pi_A)}
& IA \ar@<.5ex>[l]^{I(\eta_A)}, 
}
\]
where $F$ is the left adjoint of $G$.
The equivalence allows us to assume that for any given split extension
\[
\xymatrix{
X \ar[r]^{k} & Y \ar@<.5ex>[r]^{p}
& B \ar@<.5ex>[l]^{s}
}
\]
we can find an object $A$ in $\A$ such that the following diagram commutes
\[
\xymatrix@=3em{
JA \ar[r]^{J(\eta_A)} \ar@{=}[d] & FA \ar@<.5ex>[r]^{\varepsilon_{IA}F(\pi_A)}
\ar[d]^{\cong} & IA \ar@<.5ex>[l]^{I(\eta_A)} \ar@{=}[d] \\
X \ar[r]^{k} & Y \ar@<.5ex>[r]^{p}
& IA \ar@<.5ex>[l]^{s}
.}
\]

This fact, together with the assumption that the pair of functors is jointly conservative, 
proves that $\B$ satisfies the Split Short Five Lemma and hence any internal category object 
in $\B$ is also a internal groupoid (see \cite{Borceux-Bourn} and references there).
It remains to prove that given a internal groupoid in $\B$ we can find a Whitehead sequence such that, 
after applying the simplicial construction, the original groupoid is recovered, up to isomorphism. 

The procedure is as follows.
Given a internal groupoid as in (\ref{int-cat}), using
\[
\xymatrix{
C_1 \ar@<.5ex>[r]^{d}
& C_0 \ar@<.5ex>[l]^{e}
}
\]
we obtain an object $A$ in $\A$ such that
\[
\xymatrix@=3em{
JA \ar[r]^{J(\eta_A)} \ar@{=}[d] & FA \ar@<.5ex>[r]^{\varepsilon_{IA}F(\pi_A)}
\ar[d]^{\cong} & IA \ar@<.5ex>[l]^{I(\eta_A)} \ar@{=}[d] \\
X \ar[r]^{k} & C_2 \ar@<.5ex>[r]^{d}
& C_0 \ar@<.5ex>[l]^{e}
.}
\]
The morphism $c$ gives
\[
u:A \to GIA
\]
with $u=G(c)\eta_A$ which is such that $I(u)=1$ and $J(u)=h=c \comp \ker(d)$.

In order to obtain $v:GJA\to A$ with $J(v)=1$ and $I(v)=J(u)=h$, we consider the pair $(m,1)$ as a morphism 
of points
\[
\xymatrix{
 C_2 \ar@<.5ex>[r]^{dp_2}
\ar[d]_{m} & C_0 \ar@<.5ex>[l]^{e_2 e} \ar@{=}[d] \\
C_1 \ar@<.5ex>[r]^{d}
& C_0 \ar@<.5ex>[l]^{e}
,}
\]
and transfer it, via the equivalence, from $\Pt(\B)$ to $\A$, in order to obtain, say
\[
\xymatrix{ E \ar[d]^{m^*} \\  A.}
\]
It follows that $JE=FA_h$ where $h^*:A_h\to A$ is the cartesian lifting of $h:JA \to IA$, given by $h=J(u)$ as
 defined above. This is possible because, on the one hand $JE$ is the kernel of $dp_2$, while on the other hand, 
 $FA_h$ is the pullback of $h$ along $d$.

In this way we have a morphism
\[
\xymatrix@=3em{JE=FA_h \ar[r]^-{J(m^*)} & JA}
\]
and, via the adjuntion (see Proposition \ref{prop:basicprops}(i)), we also have  a morphism
\[
\xymatrix@=5em{A_h \ar[r]^-{\rho=GJ(m^*)\eta_A} & GJA}
\]
such that $I(\rho)=1_{JA}$ and $J(\rho)=1_{JA}$. Now, using the fact that $I$ and $J$ are jointly 
conservative we conclude that $\rho$ is an isomorphism, and so we obtain the desired $v=h^* \rho^{-1}$.
This gives a Whitehead sequence
\[
\vcenter{\xymatrix{GJA \ar[r]^-{v} & A \ar[r]^-{u} & GIA }}
\]
such that, applying the simplicial construction to it, we obtain, up to isomorphism, the original groupoid as
\[
\xymatrix@=3em{
FA_1 \ar[r]^{I(\a_2)} \ar[d]_{\cong}
& FA
\ar@<1.5ex>[r]^{\pi_0)}
\ar@<-1.5ex>[r]_{I(\a_1)}
\ar[d]^{\cong}
& IA \ar@<0ex>[l]|{\iota_0} \ar@{=}[d] \\
C_2 \ar[r]^{m} &
C_1
\ar@<1.5ex>[r]^{d}
\ar@<-1.5ex>[r]_{c}
& C_0 \ar@<0ex>[l]|{e}
}
\]
and this completes the proof.
\end{proof}

\section{Conclusion}\label{sec:conclusion}

We conclude with an application of the previous result in the case where the category $\B$ is pointed and protomodular.


In general, in order to have an action-system, we can always take $\A$ to be the category of all stable and 
exact patches in $\B$. Then, for an object $A=(X,Y,k,s,p)$ as in Definition \ref{def:exactpatch}, 
we define $I(A)=B$, $J(A)=X$ and $F(A)=Y$.
Moreover, if for every object $B$ in $\B$ the diagram
\[\xymatrix{B\ar[r]^(.4){\langle 1,0\rangle}&B\times B\ar@<.5ex>[r]^(.6){\pi_2}&B\ar@<.5ex>[l]^(.4){\langle 1,1\rangle}}\]
is a stable patch (as it is always the case in a pointed protomodular category) then we have a functor 
$G$ and the system $(I,G,J)$ is an action system of $\A$ over $\B$ provided that the L-condition  holds. 
In the case when $\B$ is a protomodular category, considering the system $(I,G,J)$ as before, 
if $f\colon{A\to Y'}$ is a organic morphism  then we have $Y'\cong FA$, which is an immediate consequence of the 
Split Short Five Lemma. This means that the L-condition  can be simplified and it becomes equivalent, in this context, 
to the following condition:
\begin{quote}
Every Peiffer graph is a multiplicative graph.
\end{quote}
In the paper \cite{MF+VdL:12} it is proved that if  $\B$ is a semi-abelian category then this condition is equivalent 
to the so-called \emph{Smith is Huq} condition.

As an application of Theorem \ref{thm:2} we can state a similar result to the one presented in \cite{MF:10a}
 concerning the description of internal groupoids in a pointed protomodular category.

\begin{quote}
Let $\B$ be a pointed and protomodular category in which every Peiffer graph is a multiplicative graph. Then, 
giving an internal groupoid in $\B$ is to give an exact patch \[\xymatrix{X_{}\ar[r]^{k}&Y&B\ar[l]_{s}}\]
together with a morphism
\[h\colon{X\to B}\] such that the two dashed arrows can be inserted in the diagram
\[\xymatrix{X\ar[r]^(.4){\langle 1,0\rangle}\ar@{=}[d]_{}&X\times X\ar@{-->}[d]&X\ar[l]_(.4){\langle 1,1\rangle}\ar[d]^{h}\\X\ar@{->}[d]_{h}\ar[r]^{k}&Y\ar@{-->}[d]&B\ar[l]_{s}\ar@{=}[d]\\B\ar[r]^(.4){\langle 1,0\rangle}&B\times B&B\ar[l]_(.4){\langle 1,1\rangle}}\]
 in order to make it commutative.
\end{quote}
Further details can be found in \cite{MF:10a}, see also \cite{B1,B2}.

\end{document}